\documentclass[12pt, reqno]{amsart}
\usepackage[margin=1.25in]{geometry}        
\usepackage{graphicx}
\usepackage{amssymb}
\usepackage{mathtools}

\usepackage{aliascnt}
\usepackage{doi}
\usepackage{xcolor}

\newtheorem{theorem}{Theorem}[section]

\newaliascnt{corx}{thmx}

\aliascntresetthe{corx}

\newaliascnt{lemma}{theorem}
\newtheorem{lemma}[lemma]{Lemma}
\aliascntresetthe{lemma}

\newaliascnt{proposition}{theorem}
\newtheorem{proposition}[proposition]{Proposition}
\aliascntresetthe{proposition}

\newaliascnt{corollary}{theorem}
\newtheorem{corollary}[corollary]{Corollary}
\aliascntresetthe{corollary}

\newaliascnt{conjecture}{theorem}
\newtheorem{conjecture}[conjecture]{Conjecture}
\aliascntresetthe{conjecture}

\newaliascnt{example}{theorem}

\aliascntresetthe{example}

\newaliascnt{question}{theorem}

\aliascntresetthe{question}

\theoremstyle{definition}
\newtheorem*{definition*}{Definition}
\newtheorem*{example*}{Example}
\newtheorem*{examples*}{Examples}

\newcommand{\B}{{\mathbb B}}

\newcommand{\R}{{\mathbb R}}
\newcommand{\Rd}{{{\mathbb R}^d}}
\newcommand{\Rn}{{{\mathbb R}^n}}
\newcommand{\Sph}{{\mathbb S}}

\DeclareMathOperator{\dist}{dist}
\DeclareMathOperator{\diam}{diam}

\DeclareMathOperator{\capd}{Cap_\mathit{d}}
\DeclareMathOperator{\capp}{Cap_\mathit{p}}

\DeclareMathOperator{\capntwo}{Cap_{\mathit{n}-2}}

\newcommand{\arxiv}[1]{%
\href{https://arxiv.org/abs/#1}{ArXiv:#1}}

\def\ddefloop#1{\ifx\ddefloop#1\else\ddef{#1}\expandafter\ddefloop\fi}
\def\ddef#1{\expandafter\def\csname bb#1\endcsname{\ensuremath{\mathbb{#1}}}}
\ddefloop ABCDEFGHIJKLMNOPQRSTUVWXYZ\ddefloop

\def\ddef#1{\expandafter\def\csname c#1\endcsname{\ensuremath{\mathcal{#1}}}}
\ddefloop ABCDEFGHIJKLMNOPQRSTUVWXYZ\ddefloop

\title{Hausdorff measure and decay rate of Riesz capacity}

\author{Qiuling Fan and Richard S. Laugesen}
\email{qiuling2@illinois.edu, Laugesen@illinois.edu}
\address{University of Illinois, Urbana, IL 61801, USA}


\begin{document}

\begin{abstract}
The decay rate of Riesz capacity as the exponent increases to the dimension of the set is shown to yield Hausdorff measure. The result applies to strongly rectifiable sets, and so in particular to submanifolds of Euclidean space. For strictly self-similar fractals, a one-sided decay estimate is found. Along the way, a purely measure theoretic proof is given for subadditivity of the reciprocal of Riesz energy. 
\end{abstract}

\maketitle

\section{\bf Introduction and results}
\label{sec:intro}

The Riesz kernel $1/|x-y|^p$ allows for energy interactions more general than the electrostatic  Coulomb repulsion. The $p$-capacity generated by the kernel can be used to measure the size of a compact set in $\Rn$. How does this Riesz capacity relate to other notions of size of the set, in particular to its measure? Since the set can have dimension smaller than the ambient dimension $n$, one intends here the appropriate Hausdorff measure. 

We show for a class of sets including smooth submanifolds that as $p$ increases to the dimension of the set, the Hausdorff measure arises from the decay rate of Riesz $p$-capacity. More precisely, it is the slope of capacity raised to the power $p$. To state the result precisely, we need some definitions. 
\begin{definition*}[Riesz energy and capacity]
Consider a nonempty compact subset $E$ of $\R^n, n \geq 1$. The Riesz $p$-energy of $E$ is
\begin{equation} \label{eq:energydef}
V_p(E) = \min_\mu \int_E \! \int_E |x-y|^{-p} \, d\mu(x) d\mu(y) , \qquad p > 0 ,
\end{equation}
where the minimum is taken over all probability measures on $E$. For the empty set, define $V_p(\emptyset)=+\infty$. The Riesz $p$-capacity is 
\begin{equation} \label{eq:capdef}
\capp(E) = V_p(E)^{-1/p} .
\end{equation}
\end{definition*}
Notice the energy is positive or $+\infty$ and so the capacity is positive or zero. The energy minimum in \eqref{eq:energydef} is known to be attained by some ``equilibrium'' measure $\mu$, by an application of the Helly selection principle, in other words, by weak-$*$ compactness of the collection of probability measures on the set \cite[Lemma 4.1.3]{BHS19}. The equilibrium measure is unique if the energy is finite, although we will not need that fact. 

The capacity is positive if and only if the energy is finite. The classical Newtonian energy $V_{n-2}(E)$ and Newtonian capacity $\capntwo(E)$ arise when $n \geq 3$ and $p=n-2$. Capacity can be regarded as measuring the size of the set, since capacity increases as the set gets larger and it scales linearly under dilation, with $\capp(sE)=s \capp(E)$ when $s>0$. 

Write $\cH^d$ for $d$-dimensional Hausdorff measure, normalized to agree with Lebesgue measure when applied to subsets of $\Rd$. Hausdorff dimension is denoted ``$\dim$''. 

\subsection{Results for rectifiable sets}
A definition by Calef and Hardin \cite[Definition 1.1]{CH09} calls a set strongly rectifiable if it can be covered by  almost-flat pieces except for an omitted set of lower dimension. 
\begin{definition*}[Strongly rectifiable sets]
Let $1 \leq d \leq n$ be integers. Call a set $E\subset \bbR^n$ strongly $d$-rectifiable if for each $\epsilon>0$ there exists a finite collection of compact subsets $K_1,\ldots,K_m \subset \bbR^d$ and corresponding bi-Lipschitz functions $\varphi_i: K_i\to E$ such that:
	\begin{itemize}
		\item[\small $\bullet$] each $\varphi_i$ has bi-Lipschitz constant less than $1+\epsilon$, 
		\item[\small $\bullet$] $\cH^d(E_i \cap E_j)=0$ when $i\neq j$, 
		\item[\small $\bullet$] $\dim(F) < d$, where $F=E\backslash \cup_{i=1}^m E_i$ is the portion of $E$ not covered by the images of the $E_i$,
	\end{itemize}
and where $E_i = \varphi_i(K_i)$ is compact.
\end{definition*}
Such an $E$ necessarily has Hausdorff dimension $\leq d$, and has finite measure: 
\[
\cH^d(E)<\infty . 
\]
Examples of strongly $d$-rectifiable sets include smooth $d$-dimensional submanifolds, and also finite unions of such submanifolds provided they intersect in sets of zero $d$-dimensional measure \cite[Section 9.5]{BHS19}. The strongly rectifiable concept is most useful when $d < n$, because when $d=n$, every compact set $E$ is strongly $n$-rectifiable simply by taking $K_1$ to equal $E$ itself. 

Now we state the main result, obtaining Hausdorff measure from $p$-capacity. 
\begin{theorem}[Hausdorff measure from decay rate of Riesz capacity] \label{th:main}
Let $1 \leq d \leq n$ be positive integers. If $E\subset \bbR^n$ is compact and strongly $d$-rectifiable then
\begin{equation} \label{eq:main}
		\lim_{p\nearrow d} \frac{\capp(E)^p}{d-p} = \frac{\cH^d(E)}{|\Sph^{d-1}|} .
\end{equation}
\end{theorem}
Here $|\Sph^{d-1}|=2\pi^{d/2} / \Gamma(d/2)$ is the surface area of the unit sphere in $\Rd$. 
\begin{corollary} \label{co:zero}
$\capd(E)=0$.
\end{corollary}
Since $\capd(E)=0$ by the corollary, the left side of \eqref{eq:main} can be interpreted as a limit of difference quotients for $p$-capacity to the power $p$. Hence \eqref{eq:main} says that the Hausdorff measure is determined by the slope of $p \mapsto \capp(E)^p$ at $p=d$. See \autoref{fig:capp} for a graphical illustration. 

The corollary is known already in greater generality because every $d$-dimensional set with $\cH^d(E)<\infty$ has $\capd(E)=0$, by \cite[Theorem 4.3.1]{BHS19}. 

The full-dimensional case of the theorem ($d=n$) says for compact $E \subset \Rn$ that 
\[
\lim_{p\nearrow n} \frac{\capp(E)^p}{n-p} = \frac{\cH^n(E)}{|\Sph^{n-1}|} . 
\]
That case of the theorem was proved by Clark and Laugesen \cite[Corollary 1.2]{CL24b}. 

The proof of \autoref{th:main} for dimensions $d\le n$, is in \autoref{sec:upper} and \autoref{sec:lower}. It builds on the case $d=n$ but requires several new ingredients, as follows. Recall that when $d<n$, the set $E$ decomposes into finitely many pieces. Those pieces can intersect, which we handle in \autoref{sec:upper} by removing a local neighborhood of the intersection points in order to eliminate energy interactions between those multiple pieces in the neighborhood. For the other direction of the proof, in \autoref{sec:lower}, we globally discard interaction energies between different pieces of the set and estimate only the self-interaction energies. This apparently wasteful technique turns out to suffice because by subadditivity of the reciprocal energy (for which we give a purely measure theoretic proof), one may recombine the estimates and show that the self-interaction terms dominate in the limit. 

An alternative proof of \autoref{th:main} could be constructed using results of Calef and Hardin \cite{CH09}. Specifically, one could use the inequalities in the proof of their Theorem 1.3, along with the result of that theorem that normalized Hausdorff measure is the weak-$*$ limit of $p$-equilibrium measure as $p \nearrow d$, to establish the formula in our equation \eqref{eq:main}. Such a proof would depend on the renormalized potential theory at $p=d$ that they develop in their paper, and hence would be more involved than the  direct approach in this paper. 

\begin{example*}  The unit sphere $\Sph^d$ is strongly $d$-rectifiable in $\Rn$ whenever $d<n$, since it is smooth and $d$-dimensional. Its Riesz $p$-capacity is 
\begin{equation} \label{eq:Rieszsphere}
\capp(\Sph^d) = 
2 \! \left( \frac{\Gamma(d-p/2) \Gamma(d/2)}{\Gamma((d-p)/2) \Gamma(d)} \right)^{\! \! 1/p} , \qquad 0<p<d, 
\end{equation}
by Borodachov, Hardin and Saff \cite[Proposition 4.6.4]{BHS19} or see Landkof \cite[p.{\,}163]{L72}; here the expression in \cite{BHS19} has been manipulated using the duplication formula \cite[5.5.5]{DLMF} to arrive at formula \eqref{eq:Rieszsphere}. For example, with $d=2$, formula \eqref{eq:Rieszsphere} gives 
\[
\capp(\Sph^2) = 2(1-p/2)^{1/p} , \qquad 0<p<2,
\]
and so $\capp(\bbS^2)^p/(2-p) = 2^{p-1}\to 2=|\bbS^2|/|\bbS^1|$ as $p\nearrow 2$, which confirms \autoref{th:main} in this case. A similar calculation works for all $d\ge 1$, as illustrated in \autoref{fig:capp}.
\end{example*}

\autoref{fig:cap} plots the capacity of the sphere as a function of $p$ for the first few values of $d$. \autoref{fig:capp} then illustrates the limit in \autoref{th:main} by plotting capacity raised to the power $p$ and showing the slope at $p=d$.

\begin{figure}[p]
  \centering
  \includegraphics[height=8.5cm]{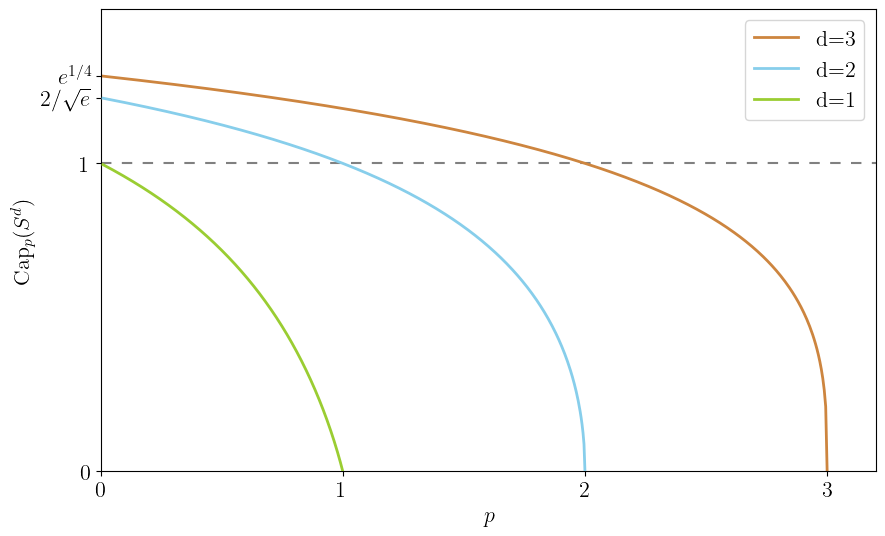}
  \caption{Plots of $p \mapsto \capp(\bbS^d)$ for spheres of dimension $d=1,2,3$, using formula \eqref{eq:Rieszsphere} from the Example. \label{fig:cap}}
\end{figure}

\begin{figure}[p]
  \centering
  \includegraphics[height=8.5cm]{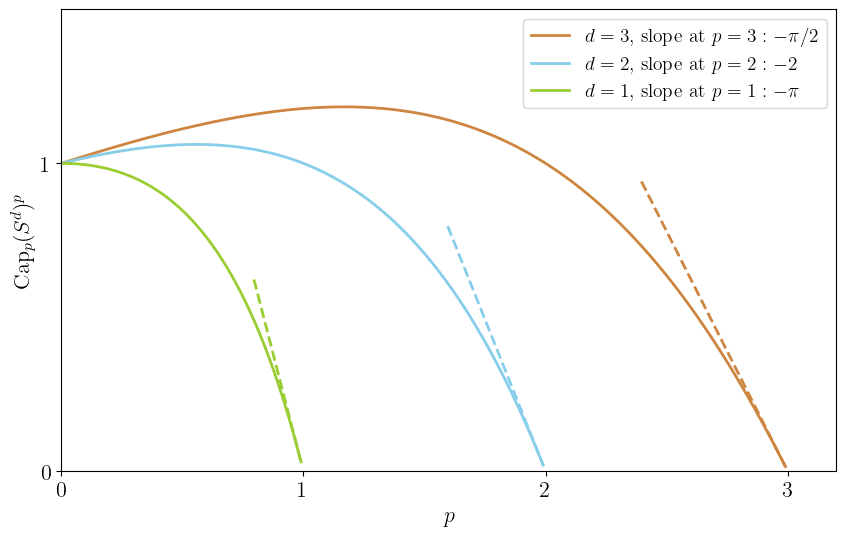}
  \caption{Plot of $p \mapsto \capp(\bbS^d)^p$ and its slope at $p=d$, for $d=1,2,3$. The plots confirm \autoref{th:main} for these spherical submanifolds. \label{fig:capp}}
\end{figure}

\subsection{Results for Ahlfors upper regular sets, including fractals}
Next we consider sets of real dimension $d$, which need not be an integer. After introducing the needed notions of density, we will relate the second order density to the limit of Riesz capacity as $p \nearrow d$. 
\begin{definition*}[Densities at dimension $d$]
Let $d > 0$. The first-order density at dimension $d$ of a finite Borel measure $\mu$ on $\bbR^n$ is	
\begin{align*}
		\rho_d(\mu,x) = \lim_{r\searrow0} \frac{\mu(\bbB^n(x,r))}{r^d} , \qquad x \in \bbR^n , 
	\end{align*}
assuming the limit exists and is finite. The second-order density is 
	\begin{align*}
		{\sigma}_{d}(\mu, x) 
		&=  \lim_{p\nearrow d} (d-p) \int_0^1 \mu(\bbB^n(x,r)) r^{-p-1}  \, dr , \qquad x \in \bbR^n , 
	\end{align*}
	again assuming the limit exists and is finite. The second-order upper density is 
	\begin{align*}
		\overline{\sigma}_{d}(\mu, x) 
		&=  \limsup_{p\nearrow d} \, (d-p) \int_0^1 \mu(\bbB^n(x,r)) r^{-p-1}  \, dr.
	\end{align*}
\qed
\end{definition*}
Excellent reference for densities and Hausdorff measure are the books by Falconer \cite{F97} and Z\"{a}hle \cite{Z24}. The second-order density can be expressed equivalently as 
	\begin{align*}
		\sigma_{d}(\mu, x) 
		&=  \lim_{\eta \searrow 0} \frac{1}{|\log \eta|}\int_{\eta}^1 \frac{\mu(\bbB^n(x,r))}{r^{d}}\, \frac{dr}{r} ,
	\end{align*}
although we will not need that formulation. The proof of this equivalence by Hinz \cite[Proposition 1.1(ii) and Theorem 1.1]{H05} can be found in Calef \cite[p.\,567]{C10}, all building on earlier work by Z\"ahle \cite[Proposition 3.1]{Z02}.

A helpful example is that Hausdorff measure on a strongly $d$-rectifiable set has constant first order density:
\begin{lemma} \label{le:rectifiabledensity}
	If $\mu$ is Hausdorff measure $\cH^d$ restricted to a strongly $d$-rectifiable set $E$, then the first-order density of $\mu$ equals $|\bbB|^d$ at $\cH^d$-almost every $x\in E$.
\end{lemma}

First order densities are stronger than second order, in the following sense.  
\begin{lemma}[see Falconer \protect{\cite[(6.22)]{F97}}] \label{le:secondfirst}
	Let $x \in \bbR^n$. If $\rho_d(\mu,x)$ exists then so does $\sigma_d(\mu,x)$, and the two numbers are equal.  
\end{lemma}
The next definition controls the rate of growth of Hausdorff measure near $x$. 
\begin{definition*}[Ahlfors upper $d$-regular set]
	Let $d>0$ and $n \geq 1$. A set $A \subset \Rn$ is said to be upper Ahlfors $d$-regular if a constant $C>0$ exists such that 
\begin{align*}
\cH^d(\bbB^n(x,r)\cap A) \leq C r^d 
\end{align*}
for all $x \in A$ and $r \in (0,\diam A]$. 
\end{definition*}
The next theorem gets a lower bound on the decay of Riesz capacity as $p \nearrow d$. It is proved in \autoref{sec:easydircproof}. 
\begin{theorem}[Second order upper density and decay of Riesz capacity] \label{th:cap_easydirc}
Let $d>0$ and $n \geq 1$. If $E \subset \bbR^n$ is compact and upper Ahlfors $d$-regular with positive $\cH^d$-measure then 
	\begin{align*}
		\liminf_{p\nearrow d} \frac{\capp(E)^p}{d-p} \ge \frac{\cH^d(E)^2/d}{\int_E \overline{\sigma}_{d}(\cH^d|_E,x) \, d\cH^d(x)} .
	\end{align*}
\end{theorem}
One would like to prove a reverse inequality on the $\limsup$, for some suitable class of sets, thus getting equality in the limit. Our attempts have not been successful.  
\begin{corollary}\label{cor:cap_easydirc}
	If the Hausdorff measure $\cH^d|_E$ in \autoref{th:cap_easydirc} has second-order density that is constant $\cH^d$-a.e., denoted $\sigma_{d}(E)$, then 
	\begin{align*}
		\liminf_{p\nearrow d} \frac{\capp(E)^p}{d-p} \ge \frac{\cH^d(E)}{d \hspace*{1pt} \sigma_{d}(E)}.
	\end{align*}
\end{corollary}
\begin{examples*}  
Suppose $E$ is a smooth submanifold of $\bbR^n$ with positive integer dimension $d$, or more generally suppose $E$ is a strongly $d$-rectifiable set. The second-order density of Hausdorff measure restricted to $E$ has the constant value $|\bbB^{d}|$, as is easily seen by approximating the submanifold locally with its tangent space. In the strongly rectifiable case, this formula follows from \autoref{le:rectifiabledensity} and \autoref{le:secondfirst}. Thus for these examples, the right side of \autoref{cor:cap_easydirc} equals $\cH^d(E)/ |\bbS^{d-1}|$, which matches the right side of \autoref{th:main}. 
\end{examples*}

\subsection*{Application to fractals}
The right side of \autoref{cor:cap_easydirc} can exceed the ``strongly rectifiable value'' $\cH^d(E)/ |\bbS^{d-1}|$ that appears on the right side of \autoref{th:main}, as we proceed to show for certain fractal sets.  

A compact set $A \subset \bbR^n$ is called a \emph{strictly self-similar fractal} if 
\begin{align*}
	A = \cup _{i=1}^N \, \varphi_i(A)
\end{align*} 
where $\varphi_i(x) = L_i U_i x + b_i$ for some $L_i \in (0,1)$, unitary matrix $U_i$, and offset $b_i \in \bbR^n$, and the sets $\{ \varphi_i(A )\}_{i=1}^N$ are disjoint. Strictly self-similar fractals possess several useful properties:
\begin{enumerate}
\item[(i)] the Hausdorff dimension $d>0$ of $A$ is determined by $\sum_{i=1}^N L_i^d = 1$ and the $d$-Hausdorff measure of $A$ is positive and finite by \cite[Thm.\ 2]{M46},
\item[(ii)] $A$ is Ahlfors $d$-regular by \cite[Lemma 3.3]{C10},
\item[(iii)] the second-order density $\sigma_{d}(\cH^d|_A, x)$ is positive, finite and constant $\cH^d$-a.e.\ by \cite[Thm.\ 1]{Z01}.
\end{enumerate}
This constant second order density value is denoted $\sigma_{d}(A)$. 

Due to these properties, \autoref{cor:cap_easydirc} immediately implies that: 
\begin{corollary}[Decay of Riesz capacity for fractals]\label{cor:decay}
	If $A$ is a strictly self-similar fractal with dimension $d$ then 
\[
		\liminf_{p\nearrow d} \frac{\capp(A)^p}{d-p} \ge \frac{\cH^d(A)}{d \hspace*{1pt} \sigma_{d}(A)}.
\]
\end{corollary}
Incidentally, \autoref{cor:cap_easydirc} applies also to self-similar sets in the sense of Z\"ahle \cite[Chapter 7]{Z24} and to self-conformal sets \cite{Z01}, since they too are known to be upper Ahlfors regular and have constant second-order density.
\begin{example*}
The middle-thirds Cantor set is a strictly self-similar fractal, as one verifies by choosing $L_1=L_2=1/3, U_1=U_2=1, b_1=0, b_2=2/3$. This set $A$ has dimension $d=(\log 2)/(\log 3)$ with $\cH^d(A)>0$ and second order density 
\[
\sigma_d(A) = 2^d (0.62344\ldots) \simeq 0.9654
\]
(see \cite[Theorem 6.6]{F97}, noting that the definition there of second order density is $2^{-d}$ times our definition). Hence for the Cantor set, the denominator on the right side of \autoref{cor:decay} is $d \hspace*{1pt} \sigma_{d}(A) \simeq 0.6091$. Meanwhile, the denominator on the right side of \autoref{th:main} is $2\pi^{d/2}/\Gamma(d/2) \simeq 1.0113$, which is larger. Hence the result of \autoref{th:main} for strongly rectifiable sets fails for some strictly self-similar fractal sets. 
\end{example*}
Does equality hold in \autoref{cor:decay}? We raise:
\begin{conjecture}\label{conj:capfractal}
	If $A$ is a strictly self-similar fractal with dimension $d$ then 
	\begin{align*}
		\lim_{p\nearrow d} \frac{\capp(A)^p}{d-p} = \frac{\cH^d(A)}{d \hspace*{1pt}\sigma_{d} (A)} .
	\end{align*}
\end{conjecture}
In support of the conjecture, note that as $p \nearrow d$, the weak-$*$ limit of $p$-equilibrium measure on the fractal set $A$ equals normalized Hausdorff measure, by Calef \cite[Theorem 1.3]{C10}. Perhaps surprisingly, that proof follows quite different lines from the corresponding work of Calef and Hardin \cite[Theorem 1.3]{CH09} for strongly rectifiable sets. The fractal arguments for convergence of the equilibrium measure do not seem to provide tools that might help prove the limit of capacity in \autoref{conj:capfractal}. Nonetheless, the research of those two authors has helped inspire the current paper. 

\subsection{Remarks}

\subsubsection*{Riesz capacity} Our definition \eqref{eq:capdef} of Riesz capacity follows Hayman and Kennedy \cite{HK76} in taking the $p$-th root of the energy, whereas other authors such as Landkof \cite{L72} do not. Another difference is that Landkof uses $n-p$ instead of $p$ as the Riesz exponent. The definition in \eqref{eq:capdef} seems the most natural choice for three reasons: it makes capacity a decreasing function of $p$, it recovers logarithmic capacity as $p \searrow 0$ (at least for nice sets), and it permits a natural extension to $p<0$. For these results see Clark and Laugesen \cite{CL24b}.  

\subsubsection*{Variational capacity} The variational capacity $\min_u \int_\Rn |\nabla u|^q \, dx$ of a set $E$, where $u \geq 1$ on $E$ and $u \to 0$ at infinity, has been studied by many authors. When $q=2$, it agrees with Newtonian capacity $\capntwo(K)$ up to a constant factor. For other values of $q$, the variational capacity does not seem to be directly connected with Riesz capacity.

\section{\bf Preliminaries on Lipschitz maps}
A Lipschitz mapping with constant $\lambda$ increases the $d$-dimensional Hausdorff measure by a factor of at most $\lambda^d$, as one sees immediately from the definitions. 
\begin{lemma}[Hausdorff measure under Lipschitz map]\label{lm1}
Let $1 \leq d \leq n$ and $\lambda>0$. If $f: \bbR^d \to \bbR^n$ is a $\lambda$-Lipschitz map then $\cH^{d}(f(A))\le \lambda^d \cH^d(A)$ for every $\cH^d$-measurable set $A\subset \bbR^d$.
\end{lemma}

The next lemma estimates a Riesz potential by ``straightening out'' the set with a bi-Lipschitz map.
\begin{lemma}[Local potential estimate] \label{lm:bound}
	Let $0\le p<d \le n$, where $d$ and $n$ are integers. If $\varphi:K\to \bbR^n$ is a bi-Lipschitz map with bi-Lipschitz constant $\lambda\ge 1$, where $K\subset \bbR^d$ is compact and $\varphi(K)=E$, then
	\begin{align*}
		\int_{E \,\cap \,\bbB^n(y,r)}\frac{1}{|x-y|^p} \, d\cH^d(x)  \le \frac{\lambda^{2d} \, r^{d-p}\,|\bbS^{d-1}|}{d-p}, \qquad y\in E,
	\end{align*}
	for all $r>0$. In particular, when $p=0$ one has 
	\begin{align*}
		\cH^d\!\left(E\cap \bbB^n(y,r)\right) \le \lambda^{2d} \, r^d \frac{|\bbS^{d-1}|}{d} , \qquad y \in E .
	\end{align*}
\end{lemma}
\begin{proof}
Fix $r>0$ and $y \in E$. By a translation of $K$, we may assume $\varphi(0)=y$. Given $x \in E$, write $x' = \varphi^{-1}(x)$ for its preimage. First we estimate the integrand, using that
\[
	|x-y|^p \ge \frac{1}{\lambda^{p}}|\varphi^{-1}(x) - \varphi^{-1}(y)|
\]
by the lower Lipschitz bound. Next, the upper Lipschitz bound says that $\varphi$ stretches each direction by at most $\lambda$ and so it increases $d$-dimensional volumes by at most $\lambda^d$. Hence by the integrand estimate and a change of variable,
\begin{align*}
	\int_{E \,\cap \,\bbB^n(y,r)}\frac{1}{|x-y|^p} \, d\cH^d(x)
	& \le\int_{E \,\cap \,\bbB^n(y,r)}\frac{\lambda^{p}}{|\varphi^{-1}(x)-\varphi^{-1}(y)|^p} \, d\cH^d(x)\\ 
	& \le \int_{\varphi^{-1}(E\,\cap\,\bbB^{n}(y, r))}\frac{\lambda^{p+d}}{|x'-0|^p} \, dx'\\ 
	&\le \int_{\bbB^{d}(0,\lambda r)}\frac{\lambda^{p+d}}{|x'|^p} \, dx'\\
	&= \frac{\lambda^{2d} \, r^{d-p}\,|\bbS^{d-1}|}{d-p},
\end{align*}
where the third inequality uses that $\varphi^{-1}(\bbB^{n}(y, r)) \subset \bbB^{d}(0,\lambda r)$, by the lower Lipschitz condition and the fact that $\varphi^{-1}(y)=0$.
\end{proof}

\section{\bf Subadditivity of reciprocal energy}\label{sec:subadditivity}

Subadditivity of the reciprocal of Riesz energy will be needed later in the paper. The standard proof relies on potential theoretic techniques \cite[p.{\,}141]{L72}, \cite[Theorem 5.28]{HK76}, perhaps because the authors aim at the better result known as strong subadditivity. Following is a short proof relying only on measure theory and Cauchy--Schwarz. 

Let $(X,\mathfrak{M})$ be a measurable space and suppose $G$ is a nonnegative, product measurable function on $X \times X$. Define the energy of a measurable set $E \subset X$ to be
\[
W(E) = \inf_\mu \int_E \int_E G(x,y) \, d\mu(x) d\mu(y)
\]  
where the infimum is taken over all probability measures $\mu$ on $E$, that is,  measures on $\mathfrak{M}$ with $\mu(E)=1$. We do not require that the infimum in the definition be attained. If $E$ does not support any probability measures, in particular if it is empty, then the energy equals $\infty$ by convention. 
\begin{proposition}[Subadditivity of reciprocal energy] \label{pr:subadditivity}
If $E_1, E_2, E_3, \ldots$ are measurable subsets of $X$ then 
\[
\frac{1}{W(\cup_i E_i)} \leq \sum_i \frac{1}{W(E_i)} .
\]
\end{proposition}
Subadditivity also holds if there are only finitely many sets $E_1,\ldots, E_m$, simply by padding the sequence with empty sets, for which the reciprocal energy equals zero.
\begin{proof}
It suffices to prove the proposition with kernel $G+\epsilon$, because that choice increases the energy of each set by $\epsilon$, after which  one may simply take $\epsilon \to 0$ in the conclusion of the proposition. Thus we may suppose from now on that the kernel $G$ and energy $W$ are bounded below by a positive constant. 

Write $E=\cup_i E_i$. If $W(E)=\infty$ then there is nothing to prove, and so we may suppose $W(E)<\infty$. Let $\mu$ be a probability measure on $E$ that has finite energy. To avoid double-counting in the proof below, we disjointify the sets: let $E_1^*=E_1, E_2^*=E_2 \setminus E_1, E_3^*=E_3 \setminus (E_1 \cup E_2)$, and so on. 

The index set $I(\mu) = \{i: \mu(E_i^*)>0\}$ is nonempty and
\[
\sum_{i\in I(\mu)} \mu(E_i^*)=\sum_{i=1}^\infty \mu(E_i^*) = \mu\left(\cup_{i=1}^\infty E_i^* \right)=\mu(E) = 1.
\]
By decomposing $E$ into the $E_i^*$ and discarding all cross terms, we estimate the energy of $\mu$ from below by
\begin{align*}
	\int_E \int_E G(x,y) \, d\mu d\mu
	& \ge \sum_{i\in I(\mu)} \int_{E_i^*} \int_{E_i^*} G(x,y) \, d\mu d\mu \\
	& \ge \sum_{i\in I(\mu)} \mu(E_i^*)^2 W(E_i),
\end{align*}
where we used that the restricted and normalized measure $\mu(\cdot \cap E_i^*)/\mu(E_i^*)$ is a probability measure on $E_i^*$ and hence also on $E_i$, and hence can serve as a trial measure for the energy $W(E_i)$. Notice that $W(E_i)$ on the right side is finite (and positive) since the left side of the inequality is finite by assumption on $\mu$. Next, by Cauchy--Schwarz, 
\begin{align*}
	\sum_{i\in I(\mu)} \mu(E_i^*)^2 W(E_i) \ge \frac{\left(\sum_{i\in I(\mu)} \mu(E_i^*)\right)^{\! 2}}{\sum_{i\in I(\mu)} W(E_i)^{-1}}
	\ge \frac{1}{\sum_{i=1}^\infty W(E_i)^{-1}}.
\end{align*}

Infimizing over the probability measures $\mu$, we deduce that 
\[
	W(E) = \inf_\mu \int_E \int_E G(x,y) \, d\mu d\mu \ge \frac{1}{\sum_{i=1}^\infty W(E_i)^{-1}} , 
\]
which proves the proposition. 
\end{proof}

\section{\bf Upper limit of the energy for \autoref{th:main}}\label{sec:upper}

The conclusion of \autoref{th:main} can be rewritten in terms of energy as 
\[
\lim_{p\nearrow d} (d-p) V_p(E) = \frac{|\Sph^{d-1}|}{\cH^d(E)} .
\]
The next proposition proves the upper direction of this equality. The lower direction is established in \autoref{sec:lower}. As usual, $d$ and $n$ are positive integers with $1 \leq d \leq n$. 
\begin{proposition} \label{pr:upper}
If $E\subset \bbR^n$ is compact and strongly $d$-rectifiable then
	\begin{equation} \label{easy_direction}
		\limsup_{p\nearrow d} (d-p) V_p(E) \le \frac{|\Sph^{d-1}|}{\cH^d(E)}.
	\end{equation}
\end{proposition}
\begin{proof}
If $\cH^d(E)=0$ then there is nothing to prove, and so we suppose $\cH^d(E)>0$. Take $0<\epsilon<\cH^d(E)$. By definition, the strongly rectifiable set $E$ decomposes as $E = \left( \cup_{i=1}^m E_i \right) \cup F$, where the set of intersection points $A = \cup_{1\le i< j \le m} (E_i\cap E_j)$ has measure zero, i.e.~$\cH^d(A) = 0$. 
Let $A(\alpha) = \{ x\in E: \dist(x,A)<\alpha\}$ be the subset of $E$ within distance $\alpha$ of $A$. Notice that $\cap_{\alpha>0}A(\alpha) = A$ since $A$ is closed. Thus by continuity of the measure from above, there exists $\alpha_0>0$ such that the set $B = A(\alpha_0)$ has $\cH^d(B)<\epsilon$. Hence $\cH^d(E \setminus B)>0$.

Let $\tilde{E}_i = E_i \backslash B$, that is, $E_i$ with the ``bad'' part $B$ removed. Notice that the sets $\tilde{E}_i$ are disjoint and compact (due to the strict inequality in the definition of $A(\alpha_0)$) and so are separated by some positive distance $\delta$, meaning $\dist(\tilde{E}_i, \tilde{E}_j)\ge \delta>0$ when $i\neq j$. 

We give two proofs of inequality \eqref{easy_direction}. The first works directly with the Riesz kernel. The second employs an alternative formula for the energy.

\subsection*{First proof} Take $\mu$ to be normalized Hausdorff measure on $E \setminus B$. That set consists of the $\tilde{E}_i$ together with $F \setminus B$, but $\cH^d(F \setminus B)=0$ and so we ignore that set in the following proof. Using $\mu$ as a trial measure in the definition of the energy, we find
\begin{align}
        V_p(E) &\le \, \frac{1}{\cH^d(E\backslash B)^2} \int_{E\backslash B} \int_{E\backslash B} \frac{1}{|x-y|^p} \, d\cH^d(x) d\cH^d(y) \notag\\
        & \le \frac{1}{\cH^d(E \backslash B)^2} \int_{E\backslash B} \int_{(E\backslash B) \setminus \bbB^n(y,\delta)} \frac{1}{|x-y|^p} \, d\cH^d(x)  d\cH^d(y) \label{int:1}\\
        & + \,\frac{1}{\cH^d(E \backslash B)^2} \sum_{i=1}^m \int_{\tilde{E}_i} \int_{(E \setminus B) \cap \, \bbB^n(y,\delta)} \frac{1}{|x-y|^p} \, d\cH^d(x)  d\cH^d(y). \label{int:2} 
\end{align}
Expression \eqref{int:1} is bounded straightforwardly by $\delta^{-p}$, since $x \notin \bbB^n(y,\delta)$ forces $|x-y| \geq \delta$. For the inner integral in \eqref{int:2}, we have $y\in \tilde{E}_i$ and so the ball $\bbB^n(y,\delta)$ does not intersect $\tilde{E}_j$ for $j \neq i$. Hence $(E \setminus B) \cap \bbB^n(y,\delta) = \tilde{E}_i \cap \bbB^n(y,\delta)$. Thus the inner integral can be estimated using \autoref{lm:bound} by 
\begin{align*}
	\int_{\tilde{E}_i \,\cap \,\bbB^n(y,\delta)}\frac{1}{|x-y|^p} \, d\cH^d(x) 
	\le \frac{(1+\epsilon)^{2d} \, \delta^{d-p}\,|\bbS^{d-1}|}{d-p} .
\end{align*}
Hence line \eqref{int:2} is bounded by
\begin{align*}
	\frac{1}{\cH^d(E\backslash B)^2} \sum_{i=1}^m \cH^d(\tilde{E}_i) \frac{(1+\epsilon)^{2d} \, \delta^{d-p}\,|\bbS^{d-1}|}{d-p}
	\le \frac{(1+\epsilon)^{2d}\,\delta^{d-p}}{d-p}\frac{|\bbS^{d-1}|}{\cH^d(E\backslash B)}.
\end{align*}

Combining the estimates on \eqref{int:1} and \eqref{int:2} and multiplying by $d-p$, we find
\begin{align*}
	(d-p)V_p(E) \le \frac{(1+\epsilon)^{2d}\,\delta^{d-p}|\bbS^{d-1}|}{\cH^d(E\backslash B)}
	+ (d-p) \delta^{-p} .
\end{align*}
Letting $p\nearrow d$ gives 
\begin{align*}
	\limsup_{p\nearrow d}(d-p) V_p(E)\le \frac{(1+\epsilon)^{2d}|\bbS^{d-1}|}{\cH^d(E\backslash B)}.
\end{align*}
Finally, recalling that $\cH^d(B) < \epsilon$ and letting $\epsilon\to 0$, we conclude 
\begin{align*}
	\limsup_{p \nearrow d}(d-p)V_p(E) \le \frac{|\bbS^{d-1}|}{\cH^d(E)}, 
\end{align*}
which is the desired estimate \eqref{easy_direction}. 

\subsection*{Second proof}
As observed by G\"{o}tz \cite[formula (3)]{G03}, the Riesz kernel can be expressed for $p>0$ as 
\[
|x-y|^{-p} = p \int_{|x-y|}^\infty r^{-p-1} \, dr = p \int_0^\infty 1_{\B^n(x,r)}(y) r^{-p-1} \, dr
\]
and so the energy becomes
\begin{equation} \label{eq:Gotz}
V_p(E) = p\min_\mu \int_0^\infty \int_E \mu(\B^n(x,r)) \, d\mu(x) \, r^{-p-1} \, dr 
\end{equation}
where the minimum is taken over probability measures on the compact set $E$. 

Choosing $\mu$ once again to be normalized Hausdorff measure on $E\backslash B$, and integrating with respect to $r$ over the intervals $(0,\delta)$ and $(\delta,\infty)$, we deduce  
\begin{align*}
	V_p(E) & \leq \frac{p}{\cH^{d}(E\backslash B)^2}
\int_0^{\delta} \int_{E\backslash B} \cH^{d}(\B^n(x,r)\cap (E\backslash B)) \, d\cH^{d}(x) \, r^{-p-1} \, dr \\
& \hspace{1.7cm} + p\int_{\delta}^\infty \int_{E} 1 \,d\mu(x)r^{-p-1} \, dr .
\end{align*}
The second term equals $\delta^{-p}$, since $\mu(E)=1$.

In the first term, when $x\in E\backslash B = \cup_{i=1}^m \tilde{E}_i$ and $r<\delta$, we know $x$ belongs to precisely one of the $\tilde{E}_i$ and $\bbB^n(x,r)$ does not intersect $\tilde{E}_j$ when $j \neq i$. By applying \autoref{lm:bound} with $p=0$ to $x \in \tilde{E}_i$, we find that 
 \begin{align*}
 	\cH^d\left(\bbB^{n}(x, r)\cap (E\backslash B)\right)
 	=  \cH^{d}(\bbB^n (x, r)\cap \tilde{E}_i)
 	\le (1+\epsilon)^{2d} r^d \frac{|\Sph^{d-1}|}{d}.
 \end{align*}
This estimate is the same for each $i$, and so we conclude
 \begin{align*}
 	V_p(E) & \le (1+\epsilon)^{2d} |\Sph^{d-1}| \frac{p/d}{\cH^d (E\backslash B)^2} \int_0^\delta \cH^d(E\backslash B) \, r^{d-p-1}  \, dr + \delta^{-p}\\
 	& = (1+\epsilon)^{2d} |\Sph^{d-1}| \frac{p/d}{\cH^d (E\backslash B)} 
 	\frac{\delta^{d-p}}{d-p} +\delta^{-p}.
 \end{align*}

Multiply both sides by $d-p$ and let $p\nearrow d$, getting that
 \begin{align*}
 	\limsup_{p\nearrow d} (d-p) V_p(E)\le (1+\epsilon)^{2d} \frac{|\bbS^{d-1}|}{\cH^d (E\backslash B)},
 \end{align*}
Let $\epsilon \to 0$ to obtain as wanted for \eqref{easy_direction} that 
\begin{equation*}
		\limsup_{p\nearrow d} (d-p) V_p(E) \le \frac{|\Sph^{d-1}|}{\cH^d(E)}.
\end{equation*}
\end{proof}

\section{\bf Lower limit of the energy for \autoref{th:main}}\label{sec:lower}

To complete the proof of \autoref{th:main}, we establish a lower bound on the energy. 
\begin{proposition} \label{pr:lower}
Let $1 \leq d \leq n$. If $E\subset \bbR^n$ is compact and strongly $d$-rectifiable then
	\begin{equation} \label{hard_direction}
		\liminf_{p\nearrow d} (d-p) V_p(E) \ge \frac{|\Sph^{d-1}|}{\cH^d(E)}.
	\end{equation}
\end{proposition}
The inequality is known already when $E$ is flat, meaning $E \subset \Rd$, by recent work of Clark and Laugesen \cite[Corollary 1.2]{CL24b}. This flat case provides a key ingredient in the following proof. 
\begin{proof}
Let $\epsilon>0$. By definition of strong rectifiability, the set partitions as $E = \left( \cup_{i=1}^m E_i \right) \dot{\cup} F$, with $E_i = \varphi_i(K_i)$ for some compact $K_i\subset \bbR^{d}$ and corresponding bi-Lipschitz function $\varphi_i$ with bi-Lipschitz constant $\le 1+\epsilon$. The intersections have vanishing Hausdorff measure: $\cH^d(E_i \cap E_j)=0$ when $i \neq j$. The set $F$ is lower dimensional, with $\dim(F) < d$, and so 
\[
\cH^d(F)=0 .
\] 

Suppose  $\dim(F) < p < d$. Because $p$ exceeds the dimension of $F$, we know by \cite[Theorem 4.3.1]{BHS19} that every compact subset of $F$ has $p$-capacity zero, that is, has infinite $p$-energy. We show now that $F$ itself (which need not be compact) has infinite $p$-energy in the sense that $\int_F \int_F |x-y|^{-p} \, d\mu d\mu=\infty$ whenever $\mu$ is a probability measure on $F$. For suppose this energy integral is finite; the compact set $F_\eta = \{ x \in E : \dist(x,\cup_{i=1}^m E_i) \geq \eta \} \subset F$ exhausts $F$ as $\eta \searrow 0$ and so $\mu(F_\eta)>0$ for some $\eta$, implying finiteness of the energy integral for the probability measure $\mu(\cdot \cap F_\eta)/\mu(F_\eta)$. Hence $\capp(F_\eta)>0$, which we already observed is not true. Therefore $F$ must have infinite $p$-energy.  

First we prove the proposition for each $E_i$ individually, up to an $\epsilon$-dependent factor. Let $\mu_i$ be a probability measure on $E_i$, so that $\mu = \mu_i \circ \varphi_i$ is a probability measure on $K_i$. Then
\begin{align*}
	V_p(E_i) &= \min_{\mu_i}\int_{E_i} \int_{E_i} \frac{1}{|\tilde{x}-\tilde{y}|^p} \, d\mu_i(\tilde{x}) d\mu_i(\tilde{y})\\
	&= \min_{\mu} \int_{K_i} \int_{K_i}  \frac{1}{|\varphi_i(x)-\varphi_i(y)|^p} \, d\mu(x) d\mu(y)\\
	& \ge \frac{1}{(1+\epsilon)^{p}} \min_{\mu} \int_{K_i} \int_{K_i} \frac{1}{|x-y|^p} \, d\mu(x) d\mu(y)\\
	&= \frac{1}{(1+\epsilon)^{p}} V_p(K_i) 
\end{align*}
where the inequality uses the upper Lipschitz condition. Now we call on the result by Clark and Laugesen \cite[Corollary 1.2]{CL24b} for $K_i\subset \bbR^{d}$, that: 
\[
\liminf_{p\nearrow d}(d-p)V_p(K_i) \ge \frac{|\Sph^{d-1}|}{\cH^{d}(K_i)} .
\]
Meanwhile, \autoref{lm1} for $\varphi_i^{-1}$ (using the lower Lipschitz condition on $\varphi_i$) gives 
\begin{align*}
	\cH^{d}(K_i) \leq (1+\epsilon)^d \cH^{d}(E_i) .
\end{align*}
Combining these inequalities, we obtain for $E_i$ that the proposition holds up to a factor of $(1+\epsilon)^{2d}$:
\begin{align}
	\liminf_{p\nearrow d}(d-p)V_p(E_i) 
	& \ge \frac{|\Sph^{d-1}|}{(1+\epsilon)^d\cH^{d}(K_i)}
	\ge \frac{|\Sph^{d-1}|}{(1+\epsilon)^{2d}\cH^{d}(E_i)}. \label{inq:E_i}
\end{align}
Notice we cannot simply let $1+\epsilon$ tend to $1$ on the right side, because the choice of $E_i$ in our decomposition depends on $\epsilon$.

Next we turn attention to the whole set $E = \left( \cup_{i=1}^m E_i \right) \dot{\cup} F$. Subadditivity of the reciprocal energy (\autoref{pr:subadditivity}) yields that 
\[
	V_p(E) \ge \frac{1}{\sum_{i=1}^m V_p(E_i)^{-1} + V_p(F)^{-1}} .
\]
We showed above that $F$ has infinite $p$-energy and so the term $V_p(F)^{-1}$ in the denominator can be dropped. 

Multiplying by $d-p$ and letting $p \nearrow d$, we see  
\begin{align*}
	\liminf_{p \nearrow d} (d-p) V_p(E)
	& \ge \frac{1}{\sum_{i=1}^m \left(\liminf_{p \nearrow d}(d-p)V_p(E_i) \right)^{-1}}\\
	& \ge \frac{|\bbS^{d-1}|}{(1+\epsilon)^{2d}}\frac{1}{\sum_{i=1}^m \cH^d(E_i)} \qquad \text{by \eqref{inq:E_i} for each $E_i$} \\
	& = \frac{|\bbS^{d-1}|}{(1+\epsilon)^{2d}}\frac{1}{\cH^d(E)} \\
	& \to \frac{|\bbS^{d-1}|}{\cH^d(E)} 
\end{align*}
as $\epsilon \to 0$, which proves the proposition. 
\end{proof}

\section{\bf Proof of \autoref{co:zero}}
$\capp(E)$ is monotonically decreasing with respect to $p$ by Clark and Laugesen \cite[Theorem 1.1]{CL24b} and so the limiting value $\lim_{p \nearrow d} \capp(E)$ exists and is greater than or equal to $\capd(E)$. That limiting value must be zero, because if it were positive then the left side of \eqref{eq:main} in \autoref{th:main} would be infinite whereas the right side is finite. Hence $\capd(E)=0$.

\section{\bf Proof of \autoref{le:rectifiabledensity}}
\label{sec:rectifiabledensityproof}

Let $\epsilon_k = 2^{-k}$ for $k\in \bbN$. For each $k$ we have a bi-Lipschitz decomposition of the strongly rectifiable set $E = \cup_i E_{k,i} \cup F_k$, given as in definition, with constant $L_k<1+\epsilon_k$. Let $A_k = \cup_{i\neq j} (E_{k,i} \cap E_{k,j})$ be the set of intersection points, which has Hausdorff measure $0$. 

Suppose $x\in E\backslash (A_k\cup F_k)$. Then $x$ belongs to only one of the $E_{k,i}$, and since those sets are compact, for sufficiently small $r$ the ball around $x$ with radius $r$ does not intersect any other $E_{k,j}$, $j\neq i$. If $\tilde x$ is the preimage of $x$ under the bi-Lipschitz mapping onto $E_{k,i}$, then
	\begin{align*}
		L_k^{-d} \,\cH^d\left(\bbB^d(\tilde x,r/L_k)\right) 
		\le \cH^d(\bbB^n(x,r)\cap E) \le L_k^d\, \cH^d\left(\bbB^d(\tilde x, L_k r)\right),
	\end{align*}
	for all small $r>0$.
	Therefore,
	\begin{equation}\label{eq:density_recset}
		\begin{split}
		\frac{|\bbB^d|}{(1+\epsilon_k)^{2d}}&\le \liminf_{r\to 0}\frac{\cH^d\left(\bbB^n(x,r)\cap E\right)}{r^d}\\
		&\le \limsup_{r\to 0}\frac{\cH^d\left(\bbB^n(x,r)\cap E\right)}{r^d} \le (1+\epsilon_k)^{2d}|\bbB^d|.
		\end{split}
	\end{equation}
	
	Let $A = \cup A_k$ and $F = \cup F_k$, so that $\cH^d(A\cup F) = 0$. Equation \eqref{eq:density_recset} holds for all $x \in E\backslash (A\cup F)$. Letting $\epsilon_k \to 0$ completes the proof.

\section{\bf Proof of \autoref{th:cap_easydirc}}
\label{sec:easydircproof}

In terms of energy, the theorem claims that  
	\begin{align*}
		\limsup_{p\nearrow d} \, (d-p) V_p(E) \le  \frac{d}{\cH^d(E)^2} \int_E \overline{\sigma}_{d}(\cH^d|_E,x) \, d\cH^d(x) .
	\end{align*}
We begin with G\"{o}tz's formula \eqref{eq:Gotz}, which says 
	\begin{align*}
		V_p(E) 
		& = p \inf_{\mu} \int_E \int_0^\infty \mu(\bbB^n(x,r)) \, r^{-p-1}\, dr \, d\mu(x) \\
		& \leq p \inf_{\mu} \int_E \int_0^1 \mu(\bbB^n(x,r)) \, r^{-p-1}\, dr \, d\mu(x) + 1 
	\end{align*}
	since $\mu(\cdot) \leq \mu(E) = 1$. Choose $\mu(\cdot) = \cH^d(\cdot \cap E) / \cH^d(E)$ to be normalized Hausdorff measure on $E$. Then  
	\begin{align*}
		V_p(E)  &\le \frac{p}{\cH^d(E)^2} \int_E \int_0^1  \cH^d(\bbB^n(x,r) \cap E) \, r^{-p-1} \,dr \, d\cH^d(x)  + 1.
	\end{align*}
	Notice that $(d-p)$ times the inner integral is dominated by the upper Ahlfors $d$-regular constant of $E$:
	\begin{align*}
		(d-p)\int_0^1  \cH^d(\bbB^n(x,r) \cap E) \, r^{-p-1} \,dr
		& \le C(d-p) \int_0^1 r^{d-p-1} \, dr = C.
	\end{align*}
	Hence 
	\begin{align*}
		& \limsup_{p\nearrow d} \, (d-p) V_p(E) \\
		&\le \frac{d}{\cH^d(E)^2} \limsup_{p\nearrow d} \int_E (d-p) \int_0^1 \cH^d(\bbB^n(x,r) \cap E) \, r^{-p-1} \,dr \, d\cH^d(x)\\
		&\leq \frac{d}{\cH^d(E)^2} \int_E \, \limsup_{p \nearrow d} \, (d-p) \int_0^1 \cH^d(\bbB^n(x,r) \cap E) \, r^{-p-1} \,dr \, d\cH^d(x)\\
		& \qquad \qquad \qquad \qquad \qquad \text{by dominated convergence} \\
		&= \frac{d}{\cH^d(E)^2} \int_E \overline{\sigma}_{d}(\cH^d|_E,x) \, d\cH^d(x) 
	\end{align*}
by definition of the second order upper density.

\section*{Acknowledgments}
Laugesen was supported by awards from the Simons Foundation (\#964018) and the National Science Foundation ({\#}2246537). The NSF grant supported Fan too. 

\bibliographystyle{plain}

\begin{thebibliography}{99}

\bibitem{BHS19}
S. V. Borodachov, D. P. Hardin and E. B. Saff, 
Discrete Energy on Rectifiable Sets. Springer Monographs in Mathematics. Springer, New York, 2019.

\bibitem{C10}
M. T. Calef, 
\emph{Riesz $s$-equilibrium measures on $d$-dimensional fractal sets as $s$ approaches $d$.} 
J. Math. Anal. Appl. 371 (2010), 564--572.

\bibitem{CH09}
M. T. Calef and D. P. Hardin, 
\emph{Riesz $s$-equilibrium measures on $d$-rectifiable sets as  $s$ approaches $d$},  
Potential Anal. 30 (2009), 385--401.

\bibitem{CL24b}
C. Clark and R. S. Laugesen, 
\emph{Riesz capacity: monotonicity, continuity, diameter and volume}. 
Preprint. \arxiv{2406.10781}

\bibitem{F97}
K. Falconer, 
Techniques in Fractal Geometry. 
John Wiley \& Sons, Ltd., Chichester, 1997.

\bibitem{G03}
M. G\"{o}tz, 
\emph{On the Riesz energy of measures}, 
J. Approx. Theory 122 (2003), 62--78.

\bibitem{HK76}
W. K. Hayman and P. B. Kennedy, 
Subharmonic Functions. Vol. I. London Mathematical Society Monographs, No.{\,} 9. Academic Press (Harcourt Brace Jovanovich, Publishers), London--New York, 1976.

\bibitem{H05}
M. Hinz, 
\emph{Average densities and limits of potentials}, 
Master’s thesis, Universit{\"a}t Jena, Jena, 2005.


\bibitem{L72}
N. S. Landkof, 
Foundations of Modern Potential Theory. Translated from the Russian by A. P. Doohovskoy. Die Grundlehren der mathematischen Wissenschaften, Band 180. Springer--Verlag, New York--Heidelberg, 1972.

\bibitem{M46}
P. A. P. Moran, 
\emph{Additive functions of intervals and Hausdorff measure}, 
Proc. Cambridge Philos. Soc. 42 (1946), 15--23.


\bibitem{Z01} 
M. Z{\"a}hle, 
\emph{The average density of self-conformal measures}, 
J. London Math. Soc. (2) 63 (2001), 721--734.

\bibitem{Z02} 
M. Z{\"a}hle, 
\emph{Forward integrals and stochastic differential equations.}
In: Seminar on Stochastic Analysis, Random Fields and Applications III.
Birkh{\"a}user Basel, 2002, pp.{\,}293--302.

\bibitem{Z24} 
M. Z{\"a}hle, 
Lectures on Fractal Geometry. 
Fractals Dyn. Math. Sci. Arts Theory Appl., 8. 
World Scientific Publishing Co. Pte. Ltd., Hackensack, NJ, 2024.

\bibitem{DLMF} NIST Digital Library of Mathematical Functions. \url{http://dlmf.nist.gov/}, Release 1.2.1 of 2024-06-15. F. W. J. Olver, A. B. Olde Daalhuis, D. W. Lozier, B. I. Schneider, R. F. Boisvert, C. W. Clark, B. R. Miller, B. V. Saunders, H. S. Cohl, and M. A. McClain, eds.

\end{thebibliography}

\end{document}